\documentclass{article}
\usepackage[utf8]{inputenc}
\usepackage{amsfonts}
\usepackage[english]{babel}
\usepackage{blindtext}
\usepackage{amssymb}
\usepackage{amsmath}
\usepackage{amsthm}
\usepackage[nottoc]{tocbibind}
\usepackage[dvipsnames]{xcolor}
\usepackage{soul}
\usepackage[hyphens]{url}
\usepackage{hyperref}
\usepackage[hyphenbreaks]{breakurl}

\newenvironment{psmallmatrix}
  {\left(\begin{smallmatrix}}
  {\end{smallmatrix}\right)}

\title{Spectral equivalence of smooth group schemes over principal ideal local rings}
\author{Itamar Hadas}
\date{April 2021}

\newtheorem{thm}{Theorem}[subsection]
\newtheorem{lemma}[thm]{Lemma}
\newtheorem{defn}[thm]{Definition}
\newtheorem{remark}[thm]{Remark}
\newtheorem{cor}[thm]{Corollary}
\newtheorem{prop}[thm]{Proposition}
\newtheorem{notation}[thm]{Notation}

\newtheorem{empt}[thm]{\empty}

\theoremstyle{plain} 
\newcommand{\thistheoremname}{}
\newtheorem{genericthm}[thm]{\thistheoremname}

\renewcommand{\quote}[1]{``#1''}

\newcommand{\embed}{\hookrightarrow}
\newcommand{\iso}{\cong}

\newcommand{\Z}{\mathbb{Z}}

\newcommand{\N}{\mathbb{N}}
\newcommand{\F}{\mathbb{F}}
\newcommand{\A}{\mathbb{A}}
\newcommand{\R}{\mathbb{R}}
\newcommand{\C}{\mathbb{C}}
\newcommand{\K}{\mathbb{K}}
\newcommand{\UU}{\mathbb{U}} 

\newcommand{\G}{\mathbb{G}}

\newcommand{\g}{\mathfrak{g}}

\newcommand{\uu}{\mathfrak{u}} 

\newcommand{\cA}{\mathcal{A}} 
\newcommand{\cD}{\mathcal{D}}
\newcommand{\cG}{\mathcal{G}}
\renewcommand{\O}{\mathcal{O}} 
\newcommand{\cX}{\mathcal{X}}

\newcommand{\Fc}{\overline{\mathbb{F}}}

\newcommand{\Fp}{\mathbb{F}_p}
\newcommand{\Fq}{\mathbb{F}_q}

\newcommand{\Zq}{\mathbb{Z}_q}
\newcommand{\Qp}{\mathbb{Q}_p}
\newcommand{\Qq}{\mathbb{Q}_q}
\newcommand{\mq}{\mathfrak{m}_q}

\newcommand{\dimirr}{\mathrm{dimirr}} 

\newcommand{\mdimirr}{\mathcal{M}_\mathrm{dimirr}}
\newcommand{\dimirrG}{\mathrm{dimirr}G}
\newcommand{\irr}{\mathrm{irr}}
\newcommand{\irrG}{\mathrm{irr}G}
\newcommand{\dimG}{{\dim G}}
\newcommand{\cmp}{\mathrm{cmp}}
\renewcommand{\ker}{\mathrm{Ker}}
\newcommand{\Gal}{\mathrm{Gal}}
\newcommand{\Hom}{\mathrm{Hom}}
\newcommand{\Mor}{\mathrm{Mor}}
\newcommand{\Out}{\mathrm{Out}}
\newcommand{\Inn}{\mathrm{Inn}}
\newcommand{\Aut}{\mathrm{Aut}}
\newcommand{\Ad}{\mathrm{Ad}}
\newcommand{\charr}{\mathrm{char}}
\renewcommand{\Im}{\mathrm{Im}}
\newcommand{\Spec}{\mathrm{Spec}}

\newcommand{\GL}{\mathrm{GL}}
\newcommand{\SL}{\mathrm{SL}}
\newcommand{\roots}{\mathrm{roots}}

\newcommand{\topn}{\langle}
\newcommand{\tc}{\rangle}

\newcommand{\Gtt}{\Tilde{G'}} 

\begin{document}
\maketitle
\tableofcontents

\begin{abstract}
Let $\cG$ be a smooth linear group scheme of finite type. For any positive integer $k$ and a finite field $\F$, let $W_k(\F)$ be the ring of Witt vectors of length $k$ over $\F$. We show that the group algebras of $\cG(\F[t]/(t^k))$ and $\cG(W_k(\F))$ are isomorphic (i.e. the multi-sets of the dimensions of the irreducible representations are equal) for any positive integer $k$ and finite field $\F$ with large enough characteristic. We also prove that if $\charr\F$ is large enough, then the cardinality of the set $\{\dim\rho\big|\rho\in \irr(\cG(\F))\}$ is bounded uniformly in $\F$. 
\end{abstract}
\maketitle
\section{Introduction}
\label{sec: intro}
\subsection{Motivation and main results}

Let $p$ be a prime number and let $q=p^l$ be a power of $p$. We will denote by $\Qq$ the unique unramified extension of $\Qp$ of degree $l$. Denote by $\Zq$ the ring of integers of $\Qq$ and by $\mq$ its maximal ideal. For any compact group $G$, denote by $\mdimirr(G)$ the multi-set of complex-valued irreducible representations of $G$. Let $\cG$ be a smooth linear group scheme of finite type. Let $k$ be a positive integer. 
In this paper, we will study the similarities between complex-valued representations of $\cG(\Fq[t]/(t^k))$ and $\cG(\Zq/\mq^k)$. In particular, we will be interested in $\mdimirr(\cG(\Fq[t]/(t^k)))$ and $\mdimirr(\cG(\Zq/\mq^k))$. Representations of such groups were already studied before. Onn \cite{Onn}, conjectured that those two multi-sets mentioned above are equal for $\cG=\GL_n$ for any positive integer $n$.\footnote{In fact, this is only a special case of his conjecture. For the full conjecture, see Conjecture 1.3 in \cite{Onn}} Singla \cite{Singla}, found a dimension-preserving bijection between the representations of both groups for $\cG=\GL_n$ and $k=2$, proving Onn's conjecture for $k=2$. Later in \cite{SinglaTwo}, she found such a bijection for other classical group such as $\mathrm{SL}_n,\mathrm{Sp}_n,\mathrm{O}_n$. Next, Stasinski and Vera-Gajardo \cite{kIsTwoForReductive} proved the equality of those multi-sets for any reductive group $\cG$, $k=2$, and assuming a certain condition on $p$. In another direction, Stasinski \cite{nIsTwo} classified all the representations of both groups when $\cG=\GL_2$. In the opposite direction, Hassain and Singla \cite{CounterExample} proved that those two multi-sets are not equal for $\cG=\mathrm{SL}_n$, $q=2$, and any even positive integer $k>2$. For more on the history of this problem, see $\cite{Singla}$ and $\cite{Onn}$. In this paper, we will prove the equality of those two multi-sets for any fixed group $\cG$ and $k$, given that the prime $p$ is large enough(where \quote{large enough} depends on $G,k$).
\newline\newline
The motivation behind this problem comes from the study of complex continuous representations of $\GL_n(\O)$ where $\O$ is the ring of integers of some non-archimedean local field $\F$. Currently, we know much more about the case when $\O=\Z_q$ for some prime-power $q$, than the case that $\O=\Fq[[t]]$ for example. Jaikin-Zapirain \cite{JaikinZapirain}, for example, studied $\mdimirr(\SL_n(\Z_q))$ using Dirichlet series, which is called ``the representation zeta function of a group'' (see Definition \ref{zeta def}). In particular he proved that, if $2\nmid q$, then the representation zeta function of $\SL_n(\Z_q)$ is of the form $\zeta_{\SL_n(\Z_q)}(s)=\sum_{i=1}^k n_i^{-s}f_i(p^{-s})$ where the $n_i$'s are integers and the $f_i$'s are rational functions\footnote{In fact, he proved that for all compact $p$-adic analytic groups $G$ with the property \quote{FAB}(which states $[H:[H,H]]$ is finite for any open subgroup $H$ of $G$)}. On the other hand, we know very little about the case where $\O=\Fq[[t]]$. If Onn's conjecture is correct for $\SL_n$, for some $n,q$ (for all $k$), then we would be able to tell much about $\SL_n(\Fq[[t]])$ (since any continuous irreducible representations of an inverse limit of finite groups $G_i$ is an irreducible representation of one of them). \newline\newline
Before stating our main results, we would need some definitions: 

\begin{defn}[Representation zeta function]
\label{zeta def}
Let $G$ be a finite group. The representation zeta function of $G$ is defined to be:
$$\zeta_{G}(s):=\sum_{\rho\in \irr(G)}\frac{1}{(\dim\rho)^{s}}$$
where $\irr(G)$ is the set of irreducible representations of $G$.
\end{defn}

\begin{defn}
\label{group equiv. def.}
Let $G_{1},G_{2}$ be two finite groups. We say they are equivalent (denoted by $G_{1}\sim G_{2}$) if one of the following equivalent conditions happens: 

1. $\C G_1 \iso \C G_2$.

2. The multiset $\{\dim\rho|\rho\in \irrG_{i}\}$ is the same for both groups.

3. $\zeta_{G_{1}}=\zeta_{G_{2}}$.

4. $\zeta_{G_{1}}(s)=\zeta_{G_{2}}(s)$ for infinitely many choices of $s\in \C$.
\end{defn}

Since the definition of equivalence depends only on the dimensions of our representations, it makes sense to use the following notations: 

\begin{notation} 
\label{dimirr def.}
Let $G$ be a finite group. Define the set: \[\dimirr(G) :=\{\dim\pi | \pi \in \irrG\}\] and denote \[N(G):=|\dimirr(G)|.\]
\end{notation}

We can now state our main results:
\begin{thm}
\label{main thm}
Let $\cG$ be a smooth linear group scheme of finite type. for any positive integer $k$, prime $p$ and $p$-power $q$, define: \[G_1(k,q):=\cG(\Fq[t]/(t^k)),G_2(k,q):=\cG(\Zq/\mq^k).\] Then for any positive integer $k$, there exists an integer $M>0$ (depending on $\cG,k$) s.t. for any prime $p>M$ and any $p$-power $q$, we have: \[G_1(k,q) \sim G_2(k,q).\] 
\end{thm}

To prove the above theorem, we will prove another interesting theorem:

\begin{thm}
\label{second big thm lighter version}
Let $\cG$ be a smooth linear group scheme of finite type. Then there exists a constant $C = C(\cG)$ s.t. for any large enough prime $p$ and any finite field $\F$ of characteristic $p$, we have $N(\cG(\F)) \leq C$.
\end{thm}

\begin{remark}
In fact, we will prove a much stronger claim that bounds the number of possible values that can appear in $\dimirr(G)$ for all linear algebraic groups $G$ with bounded degree (when considered with some embedding into $\GL_n$). See Theorem \ref{second big thm} and Remark \ref{recall cmp equiv. to deg}. 
\end{remark}

\subsection{Preliminaries, notations, and terminology}
\label{subsec: Preliminaries}
In this subsection, we establish the notations and terminology we use in this paper.
Our terminology and notations are as follows:
\begin{itemize}
    \item Binary operation on sets: Let $A,B$ be 2 sets of complex numbers. Let $*$ be a binary operation on $\C$. Then we will denote: $A*B:=\{a*b\;|\;a\in A,\;b\in B\}$
    \item Representations: As was written above, we are interested in complex-valued representations. Therefore, whenever we say ``representation'', we mean a complex-valued one.
    
    \item Restrictions and inductions: Let $G$ be an abstract group and let $H$ be a subgroup of $G$. If $\rho$ is a representation of $G$, then we will denote by $\rho\downarrow_H$ the restriction to $H$. Similarly, if $\pi$ is a representation of $H$, then we will denote by $\pi\uparrow^G$ the induction to $G$.

    \item Affine varieties: Throughout this paper, when we will write ``affine variety'' we will always assume we have some closed embedding of that variety in $\A^n$ (meaning we think of an affine variety as a zero-set of some polynomials in $\A^n$). Moreover, when we will talk about linear algebraic subgroups of $(\GL_n)_{\K}$ (for some algebraically closed field $\K$), we will take the standard embedding of $(\GL_n)_{\K}$ in $\A^{n^2+1}$ and the embedding of our group that comes with it.
    
    \item Standard notations from group theory: Let $G$ be an abstract group. We denote the center of $G$ by $Z(G)$ and the derived subgroup of $G$ by $G'$. If $G$ is also an algebraic group, we will denote by $G^0$ the connectivity component of the identity in $G$. 
    
    \item Simple algebraic groups and quasi-simple groups: We say that an algebraic group $G$, defined over a field $\F$, is simple if $G$ does not have any Zariski-closed, normal, connected, and non-trivial subgroup, that is defined over $\F$. We say that it is absolutely simple if it is simple over $\Fc$. We say that an abstract group $G$ is quasi-simple if $G/Z(G)$ is simple and $G$ is perfect (i.e. equal to its commutator subgroup). If there exists an absolutely simple and simply connected algebraic group $H$, defined over a finite field $\F$, of Lie type $L$ s.t. $G/Z(G)\iso H(\F)/Z(H(\F))$, then we say that $G$ has Lie type $L$.
\end{itemize}

\subsection{Structure of the paper}
In Section \S\ref{sec: analysis} we will use methods from basic analysis to show that to prove Theorem \ref{main thm}, it is enough to prove that $\zeta_{G_1}=\zeta_{G_2}$ on some finite number of points, which is independent of our prime $p$ (assuming Theorem \ref{second big thm lighter version}). 
\\ \\
Next, in Sections \S\ref{sec: clifford} and \S\ref{sec: algebraic geometry}, we will state and prove some key lemmas and definitions in representation theory and algebraic geometry that would help us later. In Section \S\ref{sec: algebraic groups} we will use those tools to tackle Theorem \ref{second big thm lighter version} (for sketch of the proof, see Theorem \ref{second big thm} and \ref{scketch of the proof of second big thm}).
\newline
Finally, in Section \S\ref{sec: main thm} we will prove Theorem \ref{main thm} using Theorem \ref{second big thm lighter version}, Section \S\ref{sec: analysis} and an important lemma from motivic analysis.
\subsection{Acknowledgements}
I would like to thank both Prof. Avraham Aizenbud and Prof. David Soudry for their guidance throughout the thesis research.
Additionally, I would like to mention Itay Glazer, Yotam Hendel, Elad Sayag, Guy Lachman, Shvo Regavim, and Lior Schain for fruitful conversions, ideas, references, advice, and help with proof-reading the paper.
\section{Analysis}
\label{sec: analysis}
In this section, we will prove the following proposition:
\begin{prop}
\label{equal zeta}
Let $G_1,G_2$ be two finite groups. Assume that $N(G_1)\leq k$ for some number $k$ and that $\zeta_{G_1}(s)=\zeta_{G_2}(s)$ for $4\cdot k+1$ different values of s. Then $G_1,G_2$ are equivalent. 
\end{prop}

For the proof we will need some lemmas:

\begin{lemma}
\label{calc}
Let $I\subseteq \R$ be an open interval, and let $f_n$ be a sequence of $C^1$ real functions that converges point-wise on $I$ to a continuous real function $f$ on $I$. Assume that the set of roots of $f$, denoted by $\roots(f)$, is finite and non-empty. Then we have that \[sup\{d(r,\roots(f_n\cdot f_n'))|r\in \roots(f)\}\overset{n\rightarrow \infty}{\rightarrow}0\] where d is the standard metric on $\R$ and here $d(r,\emptyset):=\infty$. 
\end{lemma}
\begin{proof}
It is enough to prove $d(r,\roots(f_n\cdot f_n'))\overset{n\rightarrow \infty}{\rightarrow}0$ for any root $r$ of $f$. Assume otherwise. Then there is an $\varepsilon$ neighborhood $U$ of $r$ s.t. there exists a subsequence of $f_n$, denoted by $f_{n_k}$, s.t. the sign of $f_{n_k},f_{n_k}'$ in $U$ is constant and independent on $k$ (for some $\varepsilon > 0$). W.L.O.G, this sub-sequence is $f_n$ itself and $f_n,f_n'>0$ for all $n$, on $U$ (because we can switch $f_n(x)\Longleftrightarrow f_n(-x)$ and $f_n(x)\Longleftrightarrow -f_n(x)$ if needed). This implies that for any $x\in U$ with $x<r$ we have $0\leq \lim_{n\to\infty}f_n(x)\leq\lim_{n\to\infty}f_n(r) = f(r) = 0$. This implies that $f(x)=0$ which contradicts the assumption that $\roots(f)$ is finite. This proves the lemma.   
\end{proof}

\begin{cor}
\label{semi polynomial cor}
Define a semi-polynomial to be a real function of the form $\sum_{k=1}^n a_k x^{\alpha_k}$ where $a_k,\alpha_k$ are real numbers with $\forall k\;a_k\neq 0$ and $n$ is a non-negative integer. Let $f$ be a semi-polynomial. Let $A$ be the list of the $a_k$'s, sorted by $\alpha_k$. Then the number of distinct positive roots of $f$ is bounded by the number of sign-switches in $A$ times $2$.
\end{cor}

\begin{proof}
Firstly, by Descartes' rule of signs, we have the corollary in the case where $f$ is a polynomial (without the \quote{times $2$}). By changing variables: $u=x^n$ for some positive integer $n$, we have the corollary in the case where all the $x$-powers in $f$ are rational (again, without the \quote{times $2$}). Now, for the general case, by multiplying by some power of $x$, we can assume W.L.O.G that all the $x$-powers are bigger than $1$. Now the claim easily follows from the rational $x$-powers case together with the previous lemma (for a large enough interval containing all positive roots of $f$ and bounded from above and below by positive numbers. It is clear that $\roots(f)$ is finite so we can find such interval). 
\end{proof}

We are now ready to prove Proposition \ref{equal zeta}:

\begin{proof} [Proof of Proposition \ref{equal zeta}]
Notice that if $G$ is a finite group then $\zeta_G(log(u))$ is a semi-polynomial in $u$. Now the proposition follows from Corollary \ref{semi polynomial cor} for $f=\zeta_{G_1}(log(u))-\zeta_{G_2}(log(u))$.
\end{proof}

\section{Representation theory and Clifford theory}
\label{sec: clifford}
In this section we will present some lemmas from representation theory and Clifford theory that would help us later:

\begin{thm}[Frobenius, {\cite[Lemma 3.1(ii)]{Frob}}]
\label{frob}
Let $G$ be a finite group and let $n$ be a positive integer. Then the number of solutions to the equation: $[x_1,y_1]...[x_n,y_n]=1$ in $G$ is equal to: \[|G|^{2n-1}\sum_{\pi\in\irrG}{\frac{1}{(\dim\pi)^{2n-2}}} = |G|^{2n-1}\zeta_G(2n-2).\]
\end{thm}

We now want to use the following simple lemma:
\begin{lemma}
\label{Sigma bound}
For a set $A$ of integers and a positive integer $M$, denote: $\Sigma(A,M):=\{n_1+..+n_r\;|\;\forall i\;n_i\in A,\;r\leq M\}$. Then, for any $A,M$ as before, we have: \[|\Sigma(A,M)|\leq M\cdot|A|^M.\]
\end{lemma}

\begin{lemma}
\label{N(H)~<N(G)}
Let $G,H$ be finite groups and $\phi: G \rightarrow H$ be a homomophism of abstract groups. Then we have: $\dimirr(H) \subseteq \Sigma(\dimirr(G),[H:\phi(G)])$ and $N(H) \leq N(G)^{[H:\phi(G)]}\cdot [H:\phi(G)]$, where $\Sigma$ is from the previous lemma.
\end{lemma}

\begin{proof}
Since $G/\ker\phi \iso \Im\phi$ and since any irreducible representation of $G/\ker\phi$ gives rise to an irreducible representation of $G$ with the same dimension, we have: $\dimirr(G) \supseteq \dimirr(\phi(G))$. So we may assume that $G \leq H$ and that $\phi$ is the identity map. Now, for any irreducible character (not necessarily one-dimensional) $\chi$ of $H$, since $\chi\Bar{\chi}\geq0$ we have: $$\topn\chi\downarrow_G,\chi\downarrow_G\tc_G = \frac{1}{|G|}\sum_{g\in G}\chi(g)\Bar{\chi}(g) \leq \frac{|H|}{|G|}\frac{1}{|H|}\sum_{h \in H}\chi(h)\Bar{\chi(h)} = \frac{|H|}{|G|}\topn\chi,\chi\tc_H = [H:G]. $$ That means that $\chi\downarrow_G$ decomposes to at most $[H:G]$ irreducible characters. Now the first part of the claim readily follows and the second part follows from the previous lemma.
\end{proof}

\begin{defn}
Let $G$ be a finite group and let $H$ be a normal subgroup of G. Let $\chi$ be an irreducible character of $H$. The inertia subgroup of $\chi$ is defined to be: $$T(\chi):=\{g\in G\;s.t.\;\chi^g=\chi \}.$$
\end{defn}

\begin{thm}[Fong's first reduction, {\cite[Proposition 3.13.2]{fong}}]
\label{fong's first reduction}
Let $G$ be a finite group and let $H$ be a normal subgroup of $G$. Let $W$ be an irreducible representation of $G$ and let $\chi$ be the character of $W$. If $$W\uparrow^{T(\chi)} = V_1\oplus...\oplus V_n$$ is the decomposition of $W\uparrow^{T(\chi)}$ to irreducible representations, then for each $V_i$ we have that $V_i\uparrow^G$ is an irreducible representation of $G$ and $V_i\uparrow^G \iso V_j\uparrow^G$ iff $V_i \iso V_j$.
\end{thm}

Let $\mu_{p^\infty}$ be the group of roots of unity in $\C$ with an order that is a power of $p$. 
\begin{lemma}[{\cite[Lemma 9.0.2]{BM}}]
\label{H^2 is good}
Let $G$ be a finite group, $H \subseteq G$ be a normal $p$-subgroup, and $\rho \in \irr(G)$. Assume that $\rho\downarrow_H$ is isotypic and that $H^2(G/H,\mu_{p^\infty})$ is trivial (when we think of $\mu_{p^\infty}$ as a trivial $G/H$ module). Then there exist irreducible representations $\pi_1, \pi_2 \in \irr(G)$ such that $\pi_1\downarrow_H$ is irreducible, $H\subseteq \ker(\pi_2)$ and $\rho \iso \pi_1 \otimes \pi_2$.
\end{lemma}

\begin{thm}[{\cite[Theorem 3.10]{finiteOrderElement}}]
\label{finite order in GLn(Z)}
Let $f(n)$ denote the maximal finite order of elements in $\GL_n(\Z)$. Then: 
$$\lim_{n\to \infty} \frac{\log(f(n))}{\sqrt{n\log n}} = 1.$$ In particular, $f(n)$ is finite for all $n$.
\end{thm}

\section{Algebraic Geometry}
\label{sec: algebraic geometry}
In this section, we will state some claims from algebraic geometry that will help us later. Recall that in our notations, when we say \quote{affine variety}, we mean a zero set of some ideal of polynomials in $\A^n$ (see Subsection \ref{subsec: Preliminaries}).

\subsection{Complexity}
In this subsection, we introduce the notion of complexity of affine algebraic varieties and present some lemmas. 

\begin{defn}[complexity]
\label{complexity def.}
Let $\K$ be an algebraically closed field and let $V$ be an affine algebraic variety in $\A_{\K}^n$ for some positive integer $n$. We will define the complexity of $V$ (denoted by $\cmp(V)$) to be the smallest positive integer $M$ s.t. $M \geq n$ and the ideal of functions vanishing on $V$ is generated by at most $M$ polynomials with a degree at most $M$. If $X, Y$ are affine varieties and $\phi: X\to Y$ is a morphism of affine varieties then we define the complexity of $\phi$ to be the complexity of its graph. If $\K$ is not algebraically closed and $V$ is an affine variety defined over $\K$ with a $\K$-embedding to $\overline{\K}^n$, then we will define the complexity of $V$ to be the complexity of $V$ when considered as a $\overline{\K}$-variety.
\end{defn}

The following lemma is one of the reasons we want to use complexity in this paper:
\begin{lemma}[{\cite[Lemma 11]{ultralimit}}]
\label{cmp(X^0)}
Let $\K$ be an algebraically closed field and let $X$ be an affine algebraic variety defined over $\K$. Assume that $X$ has complexity at most $M$. Then there exist a constant $C_{ic} = C_{ic}(M)$  s.t., $X$ has at most $C_{ic}(M)$ geometrically irreducible components and any geometrically irreducible component of $X$ has complexity at most $C_{ic}(M)$. 
\end{lemma}

Another main tool that will help us is the following lemma:

\begin{lemma}
\label{cmp(Im phi)}
Let $p$ be a prime number and let $X,Y$ be affine algebraic varieties over an algebraically closed field $\K$. Let $\phi:X\rightarrow Y$ be morphism of algebraic varieties that is defined over $\K$. Assume that $X,Y,\phi$ have complexity at most $M$. Then there exists a constant $C_{coi} = C_{coi}(M)$ s.t. $\overline{\Im\phi}$ has complexity at most $C_{coi}(M)$.
\end{lemma}

To prove this lemma, we need to introduce the notion of degree:
\begin{defn}[degree]
Let $\K$ be an algebraically closed field and let V be an affine algebraic variety in $\A_{\K}^n$ for some positive integer $n$. We will define the degree of $V$ (denoted by deg$(V)$) to be the cardinally of the intersection of $V$ with a generic affine space of dimension $n-\dim V$. If $V$ is a constructible set in $\A_{\K}^n$ (i.e., a union of Zariski-locally-closed sets), then we will define $\deg (V) := \deg(\overline{V})$.
\end{defn}

It turns out that the notion of degree and the notion of complexity are ``equivalent'', as presented in the following proposition:

\begin{prop}
\label{cmp equiv. to deg}
There exist two functions (monotonic in every variable) $C_{bd},C_{bc}:\N^2\rightarrow\N$ s.t. for any positive integers $n,M$, an algebraically closed field $\K$, and an affine algebraic variety $V$ in $\A_{\K}^n$ we have:
\begin{enumerate}
    \item If $V$ has degree less then $M$, then $V$ has complexity less the $C_{bc}(n,M)$.
    \item If $V$ has complexity less then $M$, then $V$ has degree less the $C_{bd}(n,M)$.
\end{enumerate}
\end{prop}

\begin{proof}
The first part was already proved in \cite[Theorem 9]{ultralimit}. For the second part, notice that if $W$ is an affine space s.t. $\dim(W\cap V)=0$, then, since intersecting with $W$ adds at most $n$ linear conditions, we get: $\cmp(V\cap W) \leq \cmp(V) + n$. And so by Lemma \ref{cmp(X^0)} we have that $|W\cap V|\leq C_{ic}(\cmp(V)+n)$ and thus, since $\cmp(V)<M$ by our assumption, we can take $C_{bd}(n,M):=C_{ic}(M+n)$. 
\end{proof}

The reason we want to use the notion of degree is the following lemma:
\begin{lemma}[{\cite[Lemma 2]{Degree}}]
Let $m,n$ be positive integers. Let $V$ be an affine algebraic irreducible variety in $\A^n$. Let $\phi:\A^n\rightarrow\A^m$ be an affine linear map. Then we have: $\text{deg}(V)\geq \text{deg}(\overline{\phi(V)})$.
\end{lemma} 

\begin{remark}
In \cite{Degree}, the definition of degree is quite different than in this paper(see Definition 1 and Remark 2 in \cite{Degree}). However, for irreducible algebraic varieties, the definitions are the same, so we can use the lemma in this case. 
\end{remark}

We are now ready to prove Lemma \ref{cmp(Im phi)}:
\begin{proof}[Proof of Lemma \ref{cmp(Im phi)}]
By lemma \ref{cmp(X^0)}, we can assume W.L.O.G that $X$ is irreducible. Let $V$ be the graph of $\phi$. Let $\pi$ be the projection of $V$ to $Y$. Applying the above lemmas we get:
\begin{align*}
    \cmp(\overline{\phi(X)}) = \cmp(\overline{\pi(V)}) \leq C_{bc}(M,\deg(\overline{\pi(V)})) \leq 
    C_{bc}(M,\deg(V)) \\
    \leq C_{bc}(M,C_{bd}(M,\cmp(V))) \leq C_{bc}(M,C_{bd}(M,M)).
\end{align*}
\end{proof}

\subsection{The Lang-Weil bounds} 
Another main reason that we want to use complexity here is the well-known Lang Weil bounds (\cite{langWeilOriginal}), as given here in the following form:

\begin{thm}[{\cite[Corolary 4]{lang_weil}}]
Let $\F$ be a finite field. Let $V$ be an affine variety defined over $\F$ of complexity at most $M$. Denote by $c(V)$ the number of geometrically irreducible components of $V$ that are defined over $\F$ and have the same dimension as $V$. Then:
 $$|V(\F)| = (c(V) + O_M(|\F|^{-1/2})) |\F|^{\hbox{dim}(V)}$$ where $f=O_M(g)$ if there exists a constant $C_M$, depending only on $M$ s.t. $f \leq C_M\cdot g$.
\end{thm}

From this theorem we may conclude two important corollaries:

\begin{cor}
There exists a function $C_{lw}: \N \rightarrow \N$ with the following property: \newline 
Let $n,M$ be positive integers and let $\F$ be a finite field with $|\F|>C_{lw}(M)$. Let $X_1,X_2$ be affine algebraic sub-varieties of $\A_{\Fc}^n$ defined over $\F$ with complexity at most $M$. Assume that $X_1$ is geometrically irreducible and that $\dim(X_1)>\dim(X_2)$. Then we have: $|X_1(\F)| > |X_2(\F)|$. In particular: $X_1(\F)\neq X_2(\F)$.
\end{cor}

\begin{proof}
By Lemma \ref{cmp(X^0)} and the above theorem we have:
$|X_1(\F)|=(1 + O_M(|\F|^{-1/2})) |\F|^{\hbox{dim}(X_1)}$ and $|X_2(\F)|\leq (C_{ic}(M) + O_M(|\F|^{-1/2})) |\F|^{\hbox{dim}(X_2)}$. As a consequence, if $|\F|$ is large enough then $|X_1(\F)| > |X_2(\F)|$.
\end{proof}

\begin{cor}
\label{lang weil cor}
There exist a function $C_{lw2}: \N \rightarrow \N$ with the following property: \newline
Let $n,M$ be positive integers and let $\F$ be a finite field with $|\F|>C_{lw2}(M)$. Let $X_1,X_2$ be geometrically irreducible algebraic sub-varieties of $\A_{\Fc}^n$ defined over $\F$. Assume that both $X_1,X_2$ have complexity at most $M$ and that $X_1(\F) = X_2(\F)$. Then we have: $X_1=X_2$.
\end{cor}

\begin{proof}
Define $C_{lw2}(M):=C_{lw}(2M)$. Take $M,\F,X_1,X_2$ as in the statement of the corollary. It is enough to prove that $X_1 \subseteq X_2$. Assume otherwise, then we may define $X_3:=X_1\cap X_2$. it is clear that $X_3$ is defined over $\F$ and that it has complexity at most $2M$. Notice that since $X_1 \not\subseteq X_2$, we have that $X_3$ is a proper sub-variety of $X_1$. Since $X_1$ is geometrically irreducible, we get that $X_3$ has a smaller dimension than $X_1$. But since $X_3(\F)=X_1(\F)\cap X_2(\F)=X_1(\F)$, we have a contradiction to the last corollary.
\end{proof}

\section{Algebraic Groups}
\label{sec: algebraic groups}
We first give an important definition for this section:
\begin{defn}
\label{second big thm's set}
For any positive integers $n,M_{\cmp}$, a non-negative integer $M_{\dim}$ and a finite field $\F$, define $\cA(n,M_{\dim},M_{\cmp},\F)$ to be the set of all linear algebraic groups $G$, that are defined over $\F$, are subgroups of $(\GL_n)_{\Fc}$ of dimension at most $M_{\dim}$ and complexity at most $M_{\cmp}$.
\end{defn}
In this section we will prove the following generalization of Theorem \ref{second big thm lighter version}:
\begin{thm} 
\label{second big thm}
 There exists a function $C_N:\N^3 \rightarrow \N$ s.t. for any $n,M_{\dim},M_{\cmp}$ as in Definition \ref{second big thm's set}, there exists a constant $C = C(n,M_{\dim},M_{\cmp})$ s.t. if $\F$ is a finite field of characteristic $p>C(n,M_{\dim},M_{\cmp})$, then we have: $$|\cup_{G\in\cA(n,M_{\dim},M_{\cmp},\F)}\dimirr(G(\F))| \leq C_N(n,M_{\dim},M_{\cmp}).$$ 

\end{thm}
Notice that in the statement of the theorem we may omit the assumption on the dimension of $G$, since if $G \leq (\GL_n)_{\Fc}$ then $\dimG < n^2$. The assumption is made because the proof of the theorem will be using induction on $M_{\dim}$.

\begin{remark}
\label{recall cmp equiv. to deg}
Notice that by Lemma \ref{cmp equiv. to deg} we may replace here the word ``complexity'' with the word ``degree''.
\end{remark}

\begin{remark}
Notice that if we replace ``$\F$'' with ``$\F_{p^k}$'' in Theorem \ref{second big thm}, for some fixed positive integer $k$, then we may omit the assumption that $p$ is large enough, since if $p<M$, for some fixed number $M$, then $|G(\F_{{p^k}})|\leq |\GL_n(\F_{p^k})|\leq M^{k\cdot n^2}$ ,which implies that: $\dimirrG(\F)\subseteq \{1,...,M^{k\cdot n^2}\}$.
\end{remark}

Before going into the proof of Theorem \ref{second big thm}, we give a sketch, as follows:

\begin{empt}
\label{scketch of the proof of second big thm}
Sketch of proof of the theorem: Let $G\in \cA(n,M_{\dim},M_{\cmp},\F)$. By Lemma \ref{cmp(X^0)} and Lemma \ref{N(H)~<N(G)}, we may assume $G$ is connected. If $G$ is simple then we may use the bound on $N$ for quasi-simple groups of lie type, proved in $\cite{Reductive}$. From there we use the theory of reductive groups to handle the case where $G$ is reductive. If $G$ is unipotent, then we may use Kirillov's orbit method to show that each irreducible representation of $G(\F)$ has a dimension that is a power of $|\F|$. So the interesting case is when $G$ is neither reductive nor unipotent. If $U$ is the unipotent radical of $G$, then by Fong's first reduction(Theorem \ref{fong's first reduction}), any representation $\rho$ of $G(\F)$ is either isotypic (when restricted to $U(\F)$), or induction of an irreducible character of some $T(\chi)$ (when $\chi$ is some irreducible character of $U(\F)$, and $T(\chi)$ is some proper subgroup of $G(\F)$). In the first case, we will show that we can obtain $\rho$ from representations of $G/U(\F)$ and $U(\F)$. In the second case, using Kirillov's orbit method, we will show that $T(\chi)$ is the $\F$-points of a proper algebraic subgroup of $G$ with bounded complexity. Hence, by induction, we can get what we want. 
\end{empt}

\subsection{Some lemmas}
Before we tackle Theorem \ref{second big thm}, we need to prove some lemmas: \\ \\
Note that we are dealing in the proof with $\F$-points of a quotient of algebraic groups. Therefore, the following lemma will be useful:
\begin{lemma}[{\cite[Theorem 12.3.4 \& Example 12.3.5(6)]{Springer}}]
\label{G/H(Fp)}
Let $G,H$ be two linear algebraic groups defined over a finite field $\F$ s.t. $H$ is a connected subgroup of $G$. Then we have: $(G/H)(\F) = G(\F)/H(\F)$ that is, the the natural map: $\varphi:G(\F)/H(\F)\to (G/H)(\F)$ is bijective.
\end{lemma}

Since we are dealing here with algebraic groups of bounded complexity, it will be useful to prove that the unipotent radicals of such groups are of bounded complexity. For this we need the exponent and logarithm maps:

\begin{defn}
Let $n$ be a positive integer. Let $\K$ be an algebraicly closed field of characteristic $p$ s.t. $p=0$ or $p>n$.
\begin{enumerate}
    \item For a nilpotent matrix $A \in M_n(\K)$, define: $\exp(A) = \sum_{k=0}^n \frac{A^k}{k!}$.
    \item For a unipotent matrix $A \in \GL_n(\K)$, define: $\log(A) = -\sum_{k=1}^n \frac{(I-A)^k}{k}$.
\end{enumerate}
\end{defn}

\begin{remark}
\label{exp of conjugate}
Note that by definition we have for any matrix $A$ and an invertible matrix $B$: 
\begin{enumerate}
    \item If exp is defined on $A$ then $\exp(BAB^{-1})=B\exp(A)B^{-1}$.
    \item If log is defined on $A$ then $\log(BAB^{-1})=B\log(A)B^{-1}$.
\end{enumerate}
\end{remark}

\begin{lemma}
\label{log is iso}
Let $\K$ be an algebraically closed field of characteristic $p$ and let $n < p$ be a positive integer. Let $\g$ be a nilpotent lie sub-algebra of $M_n(\K)$. Then: $\exp(\g)$ is a unipotent group and $\exp: \g \rightarrow \exp(\g)$ is an isomorphism of affine varieties with $\log$ being its inverse. In particular, $\log$ sends unipotent matrices to nilpotent ones, and $\exp$ does the opposite.
\end{lemma}

\begin{proof}
The fact that $\exp(\g)$ is a group follows from the well-known Baker Campbell Hausdorff formula. It is easy to check that log,exp are inverses of each other. Now $\exp(\g)$ is closed since if $\g$ is the zero set of the linear polynomials $f_1(x),...,f_n(x)$, then $\exp(\g)$ is the zero set of: $\{f_i(\log(x)) |i=1,..,n\}$. 
\end{proof}
\begin{remark}
Let $\F$ be a perfect sub-field of $\K$. Notice that in the notations of the lemma, since $\exp, \log$ are bijections that are defined over $\F$, we get that $\exp(\g(\F)) = \exp(\g)(\F)$ and the same is true for the $\log$ map.
\end{remark}
 We can now bound the complexity of $U$:
\begin{lemma}
\label{cmp(U) is small}
There is an increasing function $C_{cmpU} : \N \rightarrow \N$ s.t.. the following holds:
Let $n$ be a positive integer. Let $\K$ be an algebraically closed field with characteristic $p>n$. Let $G$ be a connected algebraic subgroup of $(\GL_n)_{\K}$ with complexity $M$. Let $U$ be the unipotent radical of $G$. Then $U$ has complexity at most $C_{\cmp U}(M)$.
\end{lemma}

\begin{proof}
Since $U$ is unipotent, it is conjugated to a subgroup of the unipotent upper triangular group: $\UU_n(\K)$ (see proposition 2.4.12 in \cite{Springer}). Let $A$ be the matrix that conjugates $U$, meaning $U$ is a subgroup of $A\UU_n(\K)A^{-1}$. Let $\uu$ be the Lie algebra generated by $\log(U)$. It is clear that $\uu$ is a sub-algebra of $A\log(\UU_n(\K))A^{-1}$, meaning it is nilpotent. Now, we claim that: $U = (\exp(\uu)\cap G)^0$. Indeed, since log, exp are inverses of each other, we have that $U \subseteq (\exp(\uu)\cap G)$. Since $U$ is connected, we have: $U \subseteq (\exp(\uu)\cap G)^0$. By Remark \ref{exp of conjugate}, we have that $G$ preserves $\uu$  (when acting on it by conjugation) and so we get that $\exp(\uu)\cap G$ is a normal subgroup of $G$. By Lemma \ref{log is iso}, it is also unipotent and closed. By maximality of $U$ (as a connected, normal and unipotent subgroup), we get our desired equality. Now the claim is an immediate consequence of Lemma \ref{cmp(Im phi)} (for the exponent map) and Lemma \ref{cmp(X^0)}.  
\end{proof} 

\begin{remark}
Notice that in the proof we took the Lie algebra generated by $\log(U)$. It is well-known that in characteristic zero, $\log(U)$ is already a Lie algebra for any connected unipotent group $U$. But, in positive characteristic, it may not be the case. For example, if $\K$ is an algebraicly closed field of characteristic $p>3$, then we can take $U:=\Big\{\begin{psmallmatrix}
1 & x & y\\
0 & 1 & x^p\\
0 & 0 & 1
\end{psmallmatrix}$ s.t. $x,y\in \K\Big\}$. Then $\log(U)= \Big\{\begin{psmallmatrix}
0 & x & y\\
0 & 0 & x^p\\
0 & 0 & 0
\end{psmallmatrix}$ s.t. $x,y\in \K\Big\}$, which is not even a vector space over $\K$. However, in our case, we have a bound on the complexity of $U$, so we can in fact prove that for large enough $p$, $\log(U)$ is indeed a Lie algebra. The following lemma does that:
\end{remark}

\begin{lemma}
\label{U=exp(g)}
There is an increasing function $C_{lila} : \N^2 \rightarrow \N$ s.t. for any positive integers $n,M$, an algebraically closed field $\K$ with characteristic $p>C_{lila}(n,M)$ and a unipotent subgroup $U$ of $(\GL_n)_{\K}$, with complexity at most $M$, we have that $\log(U)$ is a Lie algebra over $\K$ (which by Lemma \ref{log is iso} means that $U$ is an exponent of a Lie algebra and that it is connected). Moreover, if $U$ is defined over a perfect sub-field $\F$, then $\log(U)$ is also defined over $\F$.
\end{lemma}

\begin{proof}
The second part of this lemma is clear (assuming the first part). For the first part, we will prove that $\log(U)$ is closed under addition for large enough $p$ (the proof of the other properties of a Lie algebra is similar). Assume the contrary. Then, there is an increasing sequence of primes: $p_k$, algebraically closed fields of characteristic $p_k$: $\K_k$ and unipotent groups defined over them: $U_k$, with the conditions mentioned above, s.t. $\log(U_k)$ is not closed under addition for each $k$. That means there exist $a_k,b_k \in U_k$ s.t. \[(*)\exp(\log(a_k)+\log(b_k))\not\in U_k.\] Since each $U_k$ is conjugated to a subgroup of $\UU_n(\K_k)$, we can assume W.L.O.G that $U_k$ is a subgroup of $\UU_k(\K_k)$. Now fix some non-principal ultra-filter on the positive integers. Let $\K,U$, be the ultra-limits\footnote{For the definition and information about ultra-filters and ultra-limits, see \cite{ultralimit}.} of $\K_k,U_k$ respectively (with respect to the ultra-filter we chose). Since each $\K_k$ is algebraically closed and $p_k\rightarrow\infty$, it is not hard to see that $\K$ is an algebraically closed field of characteristic zero. Since for each $k$, the complexity of $U_k$ is bounded by $M$, it is also not hard to conclude that $U$ is an algebraic variety over $\K$ (see for example Lemma 2 in \cite{ultralimit}). Since each $U_k$ is a subgroup of $\UU_n(\K_k)$, we can also think of $U$ as a subgroup of $\UU_n(\K)$, meaning it is unipotent. But this claim is well-known for characteristic zero (see for example 
proposition 15.31 in \cite{Milne}), meaning $\log(U)$ is a Lie algebra. However, by $(*)$ we have that if we define $a,b$ to be the ultra-limits of $a_k,b_k$ respectively, then $\exp(\log(a)+\log(b))\not\in U$. A contradiction.        
\end{proof}

We now have a direct corollary of that:
\begin{cor}
\label{cardinality of a unipotent group is a power of |F|}
Let $n,M$ be positive integers. Let $\F$ be a finite field of characteristic $p>C_{lila}(n,M)$. Let $U$ be a unipotent subgroup of $(\GL_n)_{\Fc}$ that is defined over $\F$. Assume that $U$ has complexity at most $M$. Then $|U(\F)|=|\F|^k$ for some integer $0\leq k \leq n^2$.
\end{cor}

\begin{proof}
By the previous lemma, $\log(U)$ is a Lie algebra that is defined over $\F$. That means that $\log(U(\F))$ is a vector space over $\F$ of dimension $\leq n^2$. Since $\log$ is injective, we get the claim.   
\end{proof}

We would now like to state some basic lemmas on algebraic groups:
\begin{lemma}[{\cite[Lemma 3.1.4]{BM}}]
\label{bound on coker}
Let $\phi : G \to H$ be an isogeny of algebraic groups defined over a finite field $\F$. Let $K$ be its kernel. Then:
$$[H(\F):\phi(G(\F))]\leq |K(\F)|.$$
\end{lemma}

\begin{lemma}
\label{radical}
Let $\K$ be an algebraically closed field. Let $G$ be a connected reductive group defined over $\K$. Let $G'$ be the commutator group of $G$ and let $R(G)$ be the radical of $G$. Then we have:
\begin{enumerate}
    \item $R(G)$ is a central torus for $G$. 
    \item $G=R(G)\cdot G'$
    \item $R(G)\cap G'$ is finite.
\end{enumerate}
\end{lemma}

\begin{proof}
For the first claim, see Lemma 7.3.1(i) in \cite{Springer}. For the second see Corollary 8.1.6(i) in \cite{Springer}. For the third, see Lemma 7.3.1(ii) in \cite{Springer}.
\end{proof}
\begin{lemma}
\label{p doesnt divide the order}
Let $\F$ be a finite field of characteristic $p$ and let $G$ be a connected reductive group defined over $\F$. Then $p \nmid |(G/G')(\F)|$ (where $G'$ is the commutator subgroup of $G$).
\end{lemma}

\begin{proof}
Denote by $R(G)$ the radical of $G$. By Lemma \ref{radical} we have: $G = G'\cdot 
R(G)$. This implies that the map $\phi:R(G)(\Fc)\rightarrow (G/G')(\Fc)$ defined by to be the composition of the inclusion $R(G)(\Fc)\hookrightarrow G(\Fc)$ and the quotient map $G(\Fc) \twoheadrightarrow  (G/G')(\Fc)$ is onto. Since $R(G)\cap G'$ is finite (by Lemma \ref{radical}), we have that $(G/G')(\Fc)$ is isomorphic to a quotient of $R(G)(\Fc)$ by a finite group. Now, it is enough to prove that $(G/G')(\Fc)$ does not have any elements of order $p$ (since if so then $(G/G')(\F)$ also does not). Since $(G/G')(\Fc)$ isomorphic to a quotient of $R(G)(\Fc)$ by a finite group, we get that it is enough to prove that the order of any element of finite order of $R(G)(\Fc)$ isn't divisible by $p$. But by Lemma \ref{radical}, $R(G)$ is a central torus, so we have by definition: $R(G)\iso \G_m^n$ for some non-negative integer $n$. The claim now readily follows.  
\end{proof}

We can now move on to handle the two cases depending on whether the restriction of an irreducible representation of $G(\F)$ to $U(\F)$ is isotypic or not. 

\subsection{The isotypic case}
We now want to handle the first case, when the restriction of our representation to $U(\F)$ is isotypic. To do that it will be reasonable to use Lemma $\ref{H^2 is good}$. To do that we need to calculate the second cohomology of $G(\F)/U(\F)=(G/U)(\F)$ (the equality is by Lemma \ref{G/H(Fp)}). For that we have the following lemma:

\begin{lemma}
\label{H^2=1}
There is an increasing function $C_{H^2} : \N \rightarrow \N$ s.t. for any integer $n$ and connected reductive group $G$ defined over a finite field $\F$ with characteristic $p$, if $\dim G < n$ and $p > C_{H^2}(n)$, then $H^2(G(\F),\mu_{p^\infty})=1$.
\end{lemma}

To prove this lemma, we will need the following results from \cite{BM}:

\begin{thm}[{\cite[Theorem 3.1.1]{BM}}]
\label{semi simple lemma}
For any finite field $\F$ of characteristic $p>3$, and any connected, simply-connected semi-simple algebraic group $G$ defined over $\F$, we have: 
\begin{enumerate}
    \item Vanishing of $H^2$ for simply connected groups: $H^2(G(\F),A) = 1$,  for any trivial $G(\F)$-module $A$.
    \item Perfectness of simply connected groups: $G(\F)$ is perfect.
    \item Bound on the center of connected groups: $|Z(G(\Fc))|\leq 2^\dimG$.
\end{enumerate}
\end{thm}

\begin{lemma}[{\cite[Corollary 3.1.3]{BM}}]
\label{univertsal cover}
Let $\F$ be a finite field of characteristic $p>3$. Let $G$ be a semi-simple group defined over $\F$ and let $\phi:\Tilde{G}\rightarrow G$ be its universal cover. Then $\phi(\Tilde{G}(\F))=G(\F)'$.
\end{lemma}

\begin{lemma}[{\cite[Lemma 3.1.5]{BM}}]
\label{bound on the index of the derivative}
Let $G$ be a connected reductive group defined over a finite field $\F$. Assume that $char\F>3$ Then we have: $[G'(\F):G(\F)']\leq 2^{\dim(G)}$.
\end{lemma}

\begin{lemma}[{\cite[Lemma 7.0.1]{BM}}]
\label{short exat seq. and cohomology}
For any short exact sequence of finite groups: 
$$1\rightarrow\Gamma_1\rightarrow\Gamma_2\rightarrow\Gamma_3\rightarrow 1$$
we have: 
\begin{enumerate}
    \item Descent of $H^1$ to a quotient: if $p\nmid|\Gamma_1|$,then $H^i(\Gamma_2,\Fp)=H^i(\Gamma_3,\Fp)$, for all i.
    \item Descent of $H^1$ to a subgroup: if $p\nmid|\Gamma_3|$,then $H^i(\Gamma_2,\Fp)=H^i(\Gamma_1,\Fp)^{\Gamma_3}$, for all i.
\end{enumerate}
\end{lemma}

\begin{proof} [Proof of Lemma \ref{H^2=1}]
It is enough to show the claim replacing $\mu_{p^\infty}$ by $\Fp$. Define: $C_{H^2}(n):=max(3,2^n)$. Let $p>C_{H^2}(n)$. By Lemmas \ref{p doesnt divide the order} and \ref{G/H(Fp)}, we have $p \nmid [G(\F):G'(\F)]$, which implies by the descent of $H^2$ to a subgroup(Lemma \ref{short exat seq. and cohomology}(2)) that it is enough to prove: $H^2(G'(\F),\Fp)=1$. By Lemma \ref{bound on the index of the derivative} we have: $[G'(\F):G'(\F)']\leq 2^{\dim(G)} < p$, which means that it is enough to prove that $H^2(G'(\F)',\Fp)=1$. Let $\pi: \Gtt \rightarrow G'$ the universal cover of $G'$. By Lemma \ref{univertsal cover}, it is enough to prove that $H^2(\pi(\Gtt(\F)),\Fp)=1$. By the bound on the center of connected groups (Theorem \ref{semi simple lemma}(3)) and the fact that $\pi$ is a central isogeny, we have a bound on the $\F$-points of the kernel of $\pi$: $|\ker(\pi)\cap \Gtt(\F)| \leq |Z(\Gtt)| \leq 2^n < p$. By descent of $H^1$ to a quotient (Lemma \ref{short exat seq. and cohomology}(1)), it is enough to prove: $H^2(\Gtt(\F),\Fp)=1$, but that is true by vanishing $H^2$ for semi-simple groups(Theorem \ref{semi simple lemma}(1)). 

\end{proof}

We may now use Lemma \ref{H^2 is good} and the last lemma to handle the first case. The following corollary does that: 
\begin{cor}
\label{the isotypic case}
Let $\F$ be a finite field of characteristic $p$. Let $G$ be a connected linear algebraic group defined over $\F$. Let $U$ be its unipotent radical and let $\rho$ be an irreducible representation of $G(\F)$, s.t. its restriction to $U(\F)$ is isotypic. Assume $p>C_{H^2}(1+\dimG)$. Then there exist irreducible representations $\pi_1, \pi_2 \in \irr(G(\F))$ such that $\pi_1\downarrow_{U(\F)}$ is irreducible, $U(\F)\subseteq \ker(\pi_2)$ and $\rho \iso \pi_1 \otimes \pi_2$. In particular, $\dim \rho \in \dimirr(G(\F)/U(\F)) \cdot \dimirr(U(\F))$. 
\end{cor}

\begin{proof}
The claim follows immediately from Lemmas \ref{H^2 is good} and \ref{H^2=1}. 
\end{proof}

\subsection{The non-isotypic case}
We would now like to handle the second case when the restriction to $U(\F)$ is non-isotipic. To do that, we need to prove that $T(\chi)$ is the $\F$-points of some algebraic subgroup. For that, we first need to understand $\chi$, i.e. the irreducible characters of $U(\F)$. Fortunately for us, for large enough $p$, we have a nice characterization of the irreducible characters of $U(\F)$ called ``Kirillov's orbit method'': 

\begin{thm}[Kirillov's orbit method, {\cite[Theorems 1,2]{orbit}}]
Let $n$ be a positive integer and let $p>n$ be a prime number. Let $\F$ be a finite field of characteristic $p$ and let $\varphi:\F\to\C^\times$ be a non trivial character. Let $\g$ be a nilpotent lie sub-algebra of $M_n(\Fc)$ that is defined over $\F$. 
Let $G:=\exp(\g)$. Then there is a bijection between irreducible characters of $G(\F)$ and orbits $\omega$ of the coajoint action of $G(\F)$ on $\g(\F)$, sending an orbit $\omega$, to the character:
$$\chi(g):=\frac{1}{\sqrt{|\Omega|}}\sum_{\omega \in \Omega} \varphi(\omega(\log(g))).$$ 
\end{thm}

From this theorem, we have two immediate corollaries:

\begin{cor}
\label{the thm for unipotent groups}
There is an increasing function $C_{diFp}:\N^2\to \N$ with the following property:
Let $n,M$ be positive integers. Let $\F$ be a finite field of characteristic $p>C_{diFp}(n,M)$. Let $U$ be a unipotent subgroup of $(\GL_n)_{\Fc}$ that is defined over $\F$. Assume that $U$ has complexity at most $M$. Then $\dimirr (U(\F))\subseteq\{|\F|^\frac{k}{2}\big|0\leq k\leq n^2\}$.
\end{cor}

\begin{proof}
We may assume $p>C_{lila}(n,M)$. By Lemma \ref{U=exp(g)}, we get that $\uu:=\log(U)$ is a Lie algebra. Let $\chi$ be an irreducible character of $U(\F)$. Let $\Omega$ be the orbit that $\chi$ is coming from the above Theorem. We now have: \[(*)\;\chi(1)=\frac{\sum_{\omega \in \Omega}1}{\sqrt{|\Omega|}}=\sqrt{|\Omega|}.\] Now, take $f\in \Omega$. We have: \[(**)\;|\Omega|=\frac{|U(\F)|}{|Stab_{U(\F)}(f)|}=\frac{|U(\F)|}{|(Stab_{U}(f))(\F)|}.\] Since the action $\Ad^*$ of $U$ on $\uu^*$ is clearly of bounded complexity (i.e. has complexity that is bounded by some value dependent only on $n,M$), we get that $Stab_U(f)$ is of bounded complexity and therefore, for large enough $p$, we get by Lemma \ref{cardinality of a unipotent group is a power of |F|}, that $|(Stab_U(f))(\F)|=|\F|^k$ for some $0\leq k\leq n^2$. By $(*),(**)$, the claim follows.
\end{proof}

\begin{cor}
\label{orbit cor}
Let $n$ be a positive integer and let $p>n$ be a prime number. Let $\F$ be a finite field of characteristic $p$ and let $G$ be a linear algebraic subgroup of $(\GL_n)_{\Fc}$ which is defined over $\F$. Let $U$ be its unipotent radical and assume that $\uu:=\log(U)$ is a Lie algebra. Let $\chi \in \irr(U(\F))$. Then $T(\chi)$ is of the form $Stab_{G(\F)}(\Omega)$ where $\Omega$ is a coadjoint orbit of $U(\F)$ in $\uu^*$.
\end{cor}

We now have the right tools to prove what we want about $T(\chi)$.
\begin{lemma}
\label{the non-isotypic case}
Let $n, M$ be positive integers. Then there exist a constant $C_{ni}=C_{ni}(n,M)$ s.t. for large enough prime $p$, we have the following: \newline
Let $\F$ be a finite field of characteristic $p$. Let $G$ be a connected algebraic subgroup of $(\GL_n)_{\Fc}$ that is defined over $\F$, and with complexity at most $M$, and let $U$ be its unipotent radical. Let $\chi$ be an irreducible character of $U(\F)$. Then $T(\chi)$ is the $\F$-points of an algebraic subgroup of $G$ that is defined over $\F$ and has complexity at most $C_{ni}(n, M)$
\end{lemma}

\begin{proof}
Define $\uu:=\log(U)$. By Lemmas \ref{U=exp(g)} and \ref{cmp(U) is small}, for large enough $p$, we have that $\uu$ is a Lie algebra defined over $\F$. Take $f\in \uu^*(\F)$ and let $\Omega := U(\F)\cdot f$ (where the action of $U$ here is $\Ad^*$) s.t. $T(\chi)=Stab_{G(\F)}(U(\F))$ (This is possible by the previous corollary). Recall that, over algebraically closed field, any orbit of an action of a unipotent group is closed (see for example proposition 2.4.14 in \cite{Springer}). That means that $U\cdot f$ is closed. By Lemma \ref{cmp(U) is small}, we have that the complexity of $U$ is bounded by some value dependent only on $M,n$. By Lemma \ref{cmp(Im phi)} (for the map $u\mapsto u\cdot f$ from $U$ to $U\cdot f$, which is clearly of bounded complexity), we have the same for $U\cdot f$. By Lemma \ref{U=exp(g)}, for large enough $p$, we have that $Stab_U(f)$ is connected. So by Lemma $\ref{G/H(Fp)}$ we have that: \[U(\F)/Stab_U(f)(\F)=(U/Stab_U(f))(\F),\] which is equivalent to: \[U(\F) \cdot f = (U\cdot f)(\F).\] So we have \[Stab_{G(\F)}((U\cdot f)(\F)) = Stab_{G(\F)}(U(\F)\cdot f) = T(\chi).\] We claim that also: \[Stab_{G(\F)}((U\cdot f)(\F))=(Stab_G(U\cdot f))(\F).\] Indeed, the direction \quote{$\supseteq$} is clear. For the other direction, it is enough to show that for any $g\in Stab_{G(\F)}((U\cdot f)(\F))$, we have: $g\cdot (U \cdot f) = U \cdot f$. But by Corollary \ref{lang weil cor}, that is true for large enough $p$ ($U\cdot f$ and $g\cdot U\cdot f$ are geometrically irreducible varieties because they are closed and are images of maps with geometrically irreducible varieties as their domain). So overall we have $T(\chi) = (Stab_G(U\cdot f))(\F)$. Therefore, it is enough to prove that $Stab_G(U\cdot f)$ is a closed subgroup of bounded complexity. Indeed, since $U$ is a normal subgroup of $G$, we get that: $Stab_G(U\cdot f)=\{g\in G|g\cdot f\in U\cdot f\}$. Now, if $h_1(x),...,h_m(x)$ generate the ideal of functions vanishing on $U\cdot f$, then $g\mapsto h_i(g\cdot f)$ ($i\in \{1,...,m\}$) generate the ideal of functions vanishing on $Stab_G(U \cdot f)$. That proves the lemma.
\end{proof}

\subsection{The case where G is reductive}
We now want to handle the case when $G$ is reductive. To do so, we first need to bound the number of connected reductive groups of bounded dimension. For that, we will state some definitions and lemmas:

\begin{defn}(Based root datum)
Let $G$ be a reductive group defined over an algebraically closed field $\K$. Let $R$ be the root datum of $G$. Let $\Delta$ be a base of the root system of $R$. We call the pair $(R,\Delta)$ ``the based root datum'' of $G$ and the group: $\{\sigma\in Aut(R)|\sigma(\Delta)=\Delta\}$ the ``automorphism group of the based root datum'' of $G$. For the fact that the pair $(R,\Delta)$ is independent (up to isomorphism) of the choice of $\Delta$, see for example Remark 7.1.2 in \cite{root_datum}. 

\end{defn}

\begin{lemma}[{\cite[Lemma B.3.1]{BM}}]
\label{there is a finite number of complex connected reductive groups with bounded dimension}
There exist an increasing function: $C_{rn}:\N\rightarrow\N$ s.t. for any positive integer $n$ and algebraically closed field $\K$, the number of connected reductive groups over $\K$ with dimension lower then $n$ (up to an isomorphism), is less then $C_{rn}(n)$. 
\end{lemma}

\begin{remark}
In \cite{BM}, the lemma is proved for $\K = \C$, but since there exists a bijection between connected reductive groups defined over $\K$ and root data (see Theorems 9.6.2 and 10.1.1 in \cite{Springer}), and since the dimension of the connected reductive group coming from a root datum, is determined only by the root datum and not the field of definition (see corollary 8.1.3(ii) in \cite{Springer}), we get the lemma. 
\end{remark}
\begin{lemma}[{\cite[Equation (1.5.2)]{root_datum}}]
\label{Out is Aut of the based root datum}
Let $G$ be a connected reductive group defined over an algebraically closed field $\K$. Let $\Theta$ be the automorphism group of the based root datum of $G$. Then we have a short exact sequence of abstract groups: $$1\rightarrow G/Z(G)\rightarrow Aut(G) \rightarrow \Theta \rightarrow 1.$$
\end{lemma}
\begin{lemma}[{\cite[Theorem 22.50]{Milne}}]
\label{there exists a split reductive group}
Let $R$ be a root datum and let $\F$ be a field. There exist an $\F$-split connected reductive group $G$ s.t. $R$ is the root datum of $G$.
\end{lemma}

\begin{lemma}[{\cite[Lemma B.3.2]{BM}}]
\label{bound on the size of mor}
For any complex connected reductive group $G$ and any finite abelian group
$A$:
$$|\Hom(A, \Out(G))/\Ad(\Out(G))| < \infty$$ where here $\Ad(\Out(G))$ is the action of $\Out(G)$ on $\Hom(A, \Out(G))$ by conjugation and $\Out(G)$ is the group of $\C$-outer-automorphisms of $G$.
\end{lemma}

\begin{remark}
As before, we can replace $\C$ by any algebraically closed field since by Lemma \ref{Out is Aut of the based root datum}, $\Out(G)$ is the automorphism group of the based root datum of $G$ and by Lemma \ref{there exists a split reductive group}, for any root datum and a field $\F$ there exist a split connected reductive group defined over $\F$ with that root datum.
\end{remark}

\begin{lemma}[{\cite[Theorem 11.11]{lieType}}]
\label{Out(G) is finite for semisimple groups}
Let $G$ be a semi-simple algebraic group. Then $\Out(G)$ is finite.
\end{lemma}

\begin{lemma}
\label{Mor from Z^ is Mor from cyclic group}
Let $\hat{\Z}$ be the pro-finite completion of $\Z$. Then for any complex connected reductive group $G$, there is an integer $n_G > 0$
such that any continuous homomorphism $\hat{\Z} \rightarrow \Out(G)$, factors through $C_{n_G} \rightarrow \Out(G)$, where $C_{n_G}$ is the cyclic group of order $n_G$ (here $\Out(G)$ comes with the discrete topology whereas $\hat{\Z}$ comes with its usual topology).
\end{lemma}

\begin{proof}
Let $R(G)$ be the radical of $G$. by lemma \ref{radical}, we get that: $G = R(G)\cdot G'$. Now, since $R(G),G'$ are characteristic subgroups (by definition), any automorphism $\phi$ of $G$ is determent on $\phi|_{R(G)},\phi|_{G'}$. This gives us an embedding \[F:\Aut(G)\embed \Aut(R(G))\times \Aut(G').\] We claim now that $\Inn(G)$ is of finite index in $F^{-1}(\Inn(G')\times \Inn(R(G))) = F^{-1}(\Inn(G')\times \{Id\})$. Indeed, if $\phi$ is an automorphism of $G$ that fixes $R(G)\cap G'$, then, since $R(G)$ is abelian, $\phi$ is inner iff $F(\phi)$ is inner. Since $R(G)\cap G'$ is finite, we get the claim. This shows that we may replace \[\Out(G) \iso \Aut(G)/\Inn(G)\] in the lemma by \[\Aut(G)/F^{-1}(\Inn(G')\times \Inn(R(G))) \iso F(\Aut(G))/(\Inn(G')\times \Inn(R(G)))\] and we may replace that by \[(\Aut(G')\times \Aut(R(G)))/(\Inn(G')\times \Inn((R(G))) \iso \Out(G')\times \Out(R(G)).\] By Lemma \ref{Out(G) is finite for semisimple groups}, we may replace it by $\Out(R(G))$. But since $R(G)$ is a central torus, we get that: $\Out(R(G))\iso \GL_n(\Z)$ for some positive integer $n$. Now the lemma follows from Lemma \ref{finite order in GLn(Z)}.
\end{proof}

\begin{lemma}[{\cite[Example 7.2.3]{root_datum}}]
\label{H^1 is Mor}
Let $\F$ be a finite field. Denote: $\Gamma := \Gal(\Fc/\F)$. Let $G$ be a split connected reductive group defined over $\F$. Let $\Theta$ be the automorphism group of the based root datum of $G$. Then $$H^1(\F,Aut(G))=\Mor(\Gamma,\Theta)/\Ad(\Theta)$$ where here $\Mor$ means continuous morphisms where $\Gamma$ is with the usual topology and $\Theta$ is with the discrete topology. Also, $\Ad(\Theta)$ is the action of $\Theta$ on $\Mor(\Gamma, \Theta)$ by conjugation.
\end{lemma}

We can now bound the number of connected reductive groups with a bounded dimension over $\F$:
\begin{lemma}
\label{bound on the number of reductive groups}
There exist an increasing function: $C_{rn2}:\N\rightarrow\N$ s.t. for any positive integer $n$ and finite field $\F$, the number of connected reductive groups that are defined over $\F$, with dimension lower then $n$ (up to an isomorphism), is less then $C_{rn2}(n)$. 
\end{lemma}

\begin{proof}
Let $\F$ be a finite field and let $G_0$ be a connected reductive group defined over $\F$. We will denote by $X_{G_0}$, the set of equivalence classes of connected reductive groups, that are defined over $\F$ and are isomorphic to $G_0$ when considered as an $\Fc$-groups. By Lemma \ref{there is a finite number of complex connected reductive groups with bounded dimension}, it is enough to bound the cardinality of $X_{G_0}$ (independently of $\F$). Since by Lemma \ref{there exists a split reductive group}, there exists a split group in $X_{G_0}$, we may assume $G_0$ is split. It is well-known that there exist a bijection from $X_{G_0}$ to $H^1(\Gamma,Aut(G_0))$ where $\Gamma = \Gal(\Fc/\F)$(see Proposition 11.3.3 in \cite{Springer}). By Lemma \ref{H^1 is Mor}, and Lemma \ref{Out is Aut of the based root datum} it is enough to bound the size of $\Mor(\Gamma,\Out(G))/\Ad(\Out(G))$. The lemma now follows from Lemmas \ref{bound on the size of mor} and \ref{Mor from Z^ is Mor from cyclic group}(since it is well-known that $\Gamma \iso \hat{\Z}$). 
\end{proof}

We now want to prove Theorem \ref{second big thm} for connected reductive groups. We first present a similar theorem for abstract quasi-simple groups of Lie type, that is known (recall that a quasi-simple group of Lie type $L$ is an abstract perfect group $G$ with $G/Z(G)$ being simple and isomorphic to $H(\F)/Z(H(\F))$) for some absolutely simple and simply connected algebraic group $H$ of Lie type $L$. Also recall that an algebraic group over an algebraically closed field $\K$ is called absolutely simple, if it does not have any closed, normal, connected, and non-trivial subgroup):
\begin{thm}[{\cite[Theorem 1.5]{Reductive}}]
\label{the thm for quasi simple groups}
There exist a function $C_{qs}: \N \rightarrow \N$ with the following property: For any quasi-simple group of Lie type $G$ s.t. the rank of the Lie type of $G$ is $r$ (for some positive integer $r$) we have:  $N(G) \leq C_{qs}(r)$.
\end{thm}

We can then use the next theorem to prove Theorem \ref{second big thm} for absolutely simple and simply connected algebraic groups:

\begin{thm}[Tits, {\cite[Theorem 24.17]{lieType}}]
Let $G$ be an absolutely simple and simply connected algebraic group, defined over a finite field $\F$. Assume that $\charr\F>3$. Then $G(\F)$ is quasi-simple (and therefore quasi-simple of lie type).
\end{thm}

\begin{remark}
Note that in \cite{lieType}, the field of definition for all the groups is $\Fc$ meaning ``simple algebraic group'' there means ``absolutely simple algebraic group'' here.
\end{remark}

We can then use the next theorem and lemma to replace the assumption of absolute simplicity with simplicity:
\begin{thm}[{\cite[Theorem 22.121]{Milne}}]
\label{semi simple is an almost product of simple}
Let $G$ be a semi-simple algebraic group defined over an algebraically closed field $\K$. Then there are finitely many closed and simple normal subgroups: $G_1,...,G_r$ and the product map: $G_1\times G_2 \times ... \times G_r\to G$ is an isogeny.
\end{thm}
\begin{lemma}[{\cite[Proposition 22.129]{Milne}}]
Let $G,G_1,...,G_r$ be as in the previous theorem. Assume also that $G$ is defined over a perfect sub-field $\F$ and that $G$ is simple and simply connected(as an $\F$-group). Then $G(\F)\iso G_1(\F')$ where $\F'$ is the field, fixed under $\{\sigma\in \Gal(\Fc/\F)|\sigma (G_1)=G_1\}$.
\end{lemma}

\begin{remark}
Note that in \cite{Milne}, they are using the notation ``almost simple algebraic group'' group, which means ``simple algebraic group'' in our notations.
\end{remark}

\begin{remark}
Note that in the lemma, since $G$ is simply connected, we get that the map from the theorem is an isomorphism. We also get that each $G_i$ is simply connected(otherwise we may extend its universal cover to an isogeny cover of $G$). That indeed shows we can replace absolute simplicity with simplicity.
\end{remark}

We now want to bound $N(G(\F))$ for any semi-simple simply connected group $G$. 
\begin{lemma}
\label{the thm for semi simple simply connected groups}
There exist a function $C_{sssc}: \N \rightarrow \N$ s.t. for any positive integer $n$ and semi-simple, simply-connected algebraic group $G$, defined over a finite field $\F$, and with dimension less then $n$, we have: $N(G(\F)) \leq C_{sssc}(n)$.
\end{lemma}

\begin{proof}
Let $n,G,\F$ be from the given. Since $G$ is simply connected, and by Theorem \ref{semi simple is an almost product of simple}, we can write $G = G_1\times G_2 \times ... \times G_r$ (where each $G_i$ is like in Theorem \ref{semi simple is an almost product of simple}). Now, denote: $\Gamma:= \Gal(\Fc/\F)$. Then $\Gamma$ acts on the normal and simple algebraic subgroups of $G$. We can now denote $G=H_1\times H_2 \times ... \times H_k$ where each $H_i$ is the product of all the $G_j$-s in some orbit of $\Gamma$. We claim that each $H_i$ is simple (as an $\F$-group). Indeed, W.L.O.G we may assume $i=1$ and $H_1=G_1\times...\times G_l$ for some $l$. Assume $K\neq 1$ is a closed, connected, normal algebraic subgroup of $H_1$, that is defined over $\F$. For each $i$, let $\phi_i$ be the projection from $H_1$ to $G_i$. By absolute simplicity of the $G_i$-s, we then have: $\phi_i|_K$ is onto or the trivial map for all $i$. Now, since $K$ is defined over $\F$ and isn't the trivial group, we have that $\phi_i|_K$ is onto for all $i$. Now, since $K$ and $G_1$ are normal we have: $[K,G_1]\subseteq K\cap G_1$. But since all the $G_i$-s commute we have: $[K,G_1]=[\phi_1(K),G_1]=[G_1,G_1]=G_1$, so $G_1\subseteq K$ meaning $K=H_1$ like we wanted. Now, like in the previous remark, we get that each $G_i$ is simply connected and therefore also $H$. Now the claim follows from the case that $G$ is simple (using the fact that $k\leq \dim G\leq n$ and the dimension of each $H_i$ is less or equal to the dimension than $G$).
\end{proof}
We can now move on to reductive groups:
\begin{lemma}
There exist a function $C_{red}: \N \rightarrow \N$ s.t. for any positive integer $n$, a finite field $\F$ of characteristic $p>3$ and a connected reductive group $G$, defined over $\F$, with dimension at most $n$, we have: $N(G(\F))<C_{red}(n)$.
\end{lemma}

\begin{proof}
Let $n,p,G$ be as stated. Let $G'$ be the commutator group of $G$. Let $\psi:\Gtt\to G'$ be the universial cover of $G'$. Let $\phi:\Gtt \times R(G) \to G$ be the composition of the universal cover map (in the first coordinate) and the multiplication map. Let $K:=\ker\phi$. Now, notice that since $\phi$ is a composition of isogenies (since by Lemma \ref{radical}(3), the multiplication map: $G'\times R(G)\to G$ is an isogeny), it is also an isogeny. This implies that $K\subseteq Z(\Gtt \times R(G))=Z(\Gtt)\times R(G)$. Now, let $\pi_1:K\to \Gtt$ be the projection on the first coordinate. it is clear that $\pi_1$ is an injection which means by Theorem \ref{semi simple lemma}(3), that $|K(\Fc)|\leq |Z(\Gtt)(\Fc)|\leq 2^n$. By Lemma \ref{bound on coker}, we then have: $[G(\F):\phi(\Gtt\times R(G))(\F)]\leq |K(\F)|\leq |K(\Fc)| \leq 2^n$.
By Lemma \ref{N(H)~<N(G)}, it is enough to bound $N(\Gtt(\F)\times R(G)(\F))=N(\Gtt(\F))$ (since $R(G)(\F)$ is abelian). But $\Gtt$ is a semi-simple simply connected linear algebraic group, so the claim follows by Lemma \ref{the thm for semi simple simply connected groups}.

\end{proof}

\begin{cor}
\label{the thm for red. groups}
There exist a function $C_{red2}: \N \rightarrow \N$ s.t. for any positive integer $n$ and a finite field $\F$ of characteristic $p>3$, there exist a set of positive integers $S_{n,\F}$ of cardinality at most $C_{red2}(n)$ s.t. for any connected reductive group $G$, defined over $\F$, with dimension at most $n$, we have: $\dimirr(G(\F))\subseteq S_{n,\F}$.
\end{cor}

\begin{proof}
This is an immediate corollary of the above lemma and Lemma \ref{bound on the number of reductive groups}.
\end{proof}
\subsection{Proof of Theorem \ref{second big thm}}
Before we will prove the Theorem \ref{second big thm}, we want to state a useful lemma:

\begin{lemma}
\label{bound on the options of cardinalities of algebraic groups}
For any positive integers $n,M_{\cmp}$, a non-negative integer $M_{\dim}$ and a finite field $\F$, denote by $GC(n,M_{\dim},M_{\cmp},\F)$ the set of all possible cardinalities for $\F$-points of algebraic subgroups of $(\GL_n)_{\Fc}$, which are defined over $\F$, have dimension at most $M_{\dim}$ and complexity at most $M_{\cmp}$. Denote: $p:=\charr\F$. Then for large enough $p$, we have that $|GC(n,M_{\dim},M_{\cmp},\F)|$ is bounded uniformly on $\F$.

\begin{remark}
Again, like in Theorem \ref{second big thm}, we may omit $M_{\dim}$. However, $M_{\dim}$ is there to simplify the proof of this lemma and the proof of Theorem \ref{second big thm}.
\end{remark}

\begin{proof}
Let $n,M_{\dim},M_{\cmp},\F$ as given above. Let $G$ be an algebraic group with the properties mentioned above. Let $G^0$ be the connectivity component of the identity element in $G$. Let $U$ be the unipotent radical of the $G^0$. We then have by Lemma \ref{G/H(Fp)}: $|G(\F)|=[G(\F):G^0(\F)]\cdot [G^0(\F):U(\F)]\cdot |U(\F)|=|(G/G^0)(\F)|\cdot |(G^0/U)(\F)| \cdot |U(\F)|$. By Lemma \ref{cmp(X^0)}, we have: $|(G/G^0)(\F)|\leq C_{ic}(M_{cmp})$. Now, since $(G^0/U)(\F)$ is reductive, by Lemma \ref{bound on the number of reductive groups}, the number of options for $|(G^0/U)(\F)|$ is bounded by $C_{rn2}(M_{\dim})$. As for $U(\F)$, by Lemmas \ref{cmp(U) is small} and \ref{cardinality of a unipotent group is a power of |F|}, we have: $|U(\F)|=|\F|^k$ for some $0\leq k\leq n^2$. This means we have: $|GC(n,M_{\dim},M_{\cmp},\F)|\leq C_{ic}(M_{\cmp})\cdot C_{rn2}(M_{\dim}) \cdot n^2$.
\end{proof}

\end{lemma}

We are now ready to prove Theorem \ref{second big thm}:

\begin{proof} [Proof of Theorem \ref{second big thm}]
Denote: \[\cD(n,M_{\dim},M_{\cmp},\F):=\cup_{G\in\cA(n,M_{\dim},M_{\cmp},\F)}\dimirr(G(\F)).\] We will prove the claim using induction on $M_{\dim}$. Let $G \in \cA(n,M_{\dim},M_{\cmp},p)$. We want to bound the options for $\dimirr(G(\F))$. By Lemmas \ref{G/H(Fp)} and \ref{cmp(X^0)}, we have: $[G(\F):G^0(\F)]=|(G/G^0)(\F)| \leq |C_{ic}(M_{\cmp})|$. By Lemma \ref{N(H)~<N(G)}, we have(with the notations from that lemma): \[\dimirr(G(\F))\subseteq\Sigma(\dimirr(G^0(\F)),C_{ic}(M_{\cmp})).\] Therefore, we may assume that $G$ is connected. If $M_{\dim}=0$ then $G=1$ and the claim is clear. So we may assume $M_{\dim}>0$. Because of Lemma \ref{the thm for red. groups}, if $p>3$, then we may assume that $G$ is not reductive. Because of Lemma \ref{the thm for unipotent groups}, if $p$ is large enough, then we may assume $G$ is not unipotent either. Let $U\not\in \{G,\{1\}\}$ be the unipotent radical of $G$ and let $H:=G/U$. We will prove that if $p$ is large enough, then any irreducible representation of $G(\F)$ is one of the followings:
\begin{enumerate}
    \item A tensor product of a representation of $G(\F)$ with dimension in $\dimirr(U(\F))$ and a representation of $H(\F)$ (when we identify representations of $H(\F)$ with representations of $G(\F)$ with trivial $U(\F)$ action).
    
    \item An induction of an irreducible representation of $K(\F)$ where $K$ is a proper algebraic subgroup of $G$ with complexity that is bounded by $C_{ni}(n,M_{\cmp})$ (see Lemma \ref{the non-isotypic case} for the definition of $C_{ni}$).

\end{enumerate}
Indeed, let $\rho$ be an irreducible representation of $G(\F)$. By Fong's first reduction(Theorem \ref{fong's first reduction}) (applied to $G(\F),U(\F)$), $\rho\downarrow_{U(\F)}$ is isotypic, or $\rho$ is an induction of an irreducible representation of $T(\chi)$ for some irreducible character $\chi$ of $U(\F)$ and $T(\chi)\neq G(\F)$. In the first case, by Corollary \ref{the isotypic case}, this is the situation of (1). In the second case, by Lemma \ref{the non-isotypic case}, this is the situation of (2) (for large enough $p$). \\ 
Now, we may denote by $A_1$ the set of all the dimensions of irreducible representations of $G(\F)$ from the first case, and by $A_2$ the same from the second case. By Lemmas \ref{the thm for red. groups} and \ref{the thm for unipotent groups}, we have that: $A_1 \subseteq\{|\F|^\frac{k}{2}\cdot d | 0\leq k \leq n^2,\;d\in S_{M_{\dim},\F}\}$ (where $S_{M_{\dim},p}$ is the set from Lemma $\ref{the thm for red. groups}$). As for $A_2$, take $\rho$ from the second case. Then $\rho = \psi \uparrow^{G(\F)}$ for some irreducible representation $\psi \in \irr K(\F)$. We then get:
$$\dim\psi \in \cD(n,M_{\dim}-1,C_{ni}(n,M_{\cmp}),\F),$$
(where the induction assumption is made on $K$, which is a proper subgroup of $G$, so it has a lower dimension since $G$ is connected).

We then get (using the notation of Lemma \ref{bound on the options of cardinalities of algebraic groups}) that: $$\dim\rho = [G(\F):K(\F)]\cdot \dim\psi \in \frac{GC(n,M_{\dim},M_{\cmp},\F)}{GC(n,M_{\dim},C_{ni}(n,M_{\cmp}),\F)}\cdot \cD(n,M_{\dim}-1,C_{ni}(n,M_{\cmp}),\F).$$ \\ \\ 
In conclusion, we have found a set, depending only on $n,M_{\dim},M_{\cmp},\F$, with cardinality that is bounded uniformly on $\F$, that contains $\dimirr(G)$. That proves the theorem.
\end{proof}

We can now prove the weaker version of Theorem \ref{second big thm}, presented in Section \S\ref{sec: intro}:

\begin{proof}[Proof of Theorem \ref{second big thm lighter version}]
Let $\cG$ be a smooth linear group scheme of finite type. We can write $\cG=\Spec(\Z[\{x_{i,j}\}_{i,j=1}^{n}\cup \{y\}]/I)$ for some positive integer $n$ and some ideal $I$ of the ring $\Z[\{x_{i,j}\}_{i,j=1}^{n}\cup \{y\}]$ that contains the polynomial $\det((x_{i,j}))\cdot y - 1$ (since $\GL_n = \Spec(\Z[\{x_{i,j}\}_{i,j=1}^{n}\cup \{y\}]/(\det((x_{i,j}))\cdot y - 1))$). Since $\Z$ is noetherian, $I$ is finitly generated. Let $S$ be its generating set and let $M_0$ be the maximal degree of all the polynomials in $S$. Denote: $M_1:=\max\{n^2+1,|S|,M_0\}$. Then, by the definition of complexity, for any finite field $\F$ we have: $\cmp((\cG)_{\Fc}) \leq M_1$. By $(*)$ and by Theorem \ref{second big thm} we have that there exists a constant $C$, s.t for any large enough prime $p$, and any $p$-power $q$, we have: $N(\cG(\Fq))<C$, as wanted.
\end{proof}

\section{Main theorem}
\label{sec: main thm}
In this section, we will prove Theorem \ref{main thm}.
\newline
For the poof of the theorem, we will use the following lemma:
\begin{lemma}
\label{equal zeta for large p}
Let $\cG$ be a smooth linear group scheme of finite type. Let $m$ be a positive integer, or $-1$. For $i\in \{1,2\}$, a prime $p$, a $p$-power $q$, and a positive integer $k$, denote $G_i(k,q)$ as in Theorem \ref{main thm}. Then for a fixed $k$, we have for large enough $p$: $\zeta_{G_1(k,q)}(2m) = \zeta_{G_2(k,q)}(2m)$.
\end{lemma}

For the proof of Lemma \ref{equal zeta for large p} we will need the following result from motivic integration:
\begin{lemma}[{\cite[Proposition 3.0.2]{motivic}}]
Let $\cX$ be a scheme of finite type and let $k$ be a positive integer. Then for any large enough prime $p$ and any $p$-power $q$, we have: $|\cX(\Zq/\mq^k)|=|\cX(\Fq[t]/t^k)|$. 
\end{lemma}

\begin{proof} [Proof of Lemma \ref{equal zeta for large p}]
Firstly, for $m=-1$, we need to prove that \[(*)|G_1(k,q)|=|G_2(k,q)|.\] But this is just the last lemma for $\cX=\cG$. We now assume that $m>0$. By Theorem \ref{frob}, we have: \[\zeta_{G_i(k,q)}(2m) = |G_i(k,q)|^{2m-1}|\{(g_1,...,g_m,h_1,...,h_m)\in G_i(k,q)^{2m}| [g_1,h_1]...[g_m,h_m]=1\}|.\] We can the define $\cX$ to be the scheme that satisfies that for any ring $R$ we have:  \[\cX(R)=\{(g_1,...,g_m,h_1,...,h_m)\in \cG(R)^{2m}\;|\; [g_1,h_1]...[g_m,h_m]=1\}.\] We can then use the last lemma for $\cX$ and $(*)$ to get what we want.  
\end{proof}

We are now ready to prove Theorem \ref{main thm}:
\begin{proof}[Proof of Theorem \ref{main thm}]
Let $J_k(\cG)$ be the k-th jet scheme of $\cG$. Since $\cG$ is smooth linear group scheme, $J_k(\cG)$ is also smooth linear group scheme. Now, for any field $\F$, we have: $(*)\;J_k(\cG)(\F)\iso \cG(\F[t]/(t^k))\overset{\mathrm{def}}{=}G_1(q,k)$ (as abstract groups). By $(*)$ and by Theorem \ref{second big thm lighter version} we have that there exists a constant $C$, s.t for any large enough prime $p$, and any $p$-power $q$, we have: $N(G_1(k,q))<C$. By Proposition \ref{equal zeta}, we have that it is enough to show that $\zeta_{G_1(k,q)}=\zeta_{G_2(k,q)}$ on $4\cdot C+1$ distinct points. By Lemma \ref{equal zeta for large p}, we have this for large enough $p$.
\end{proof}

\medskip

\bibliographystyle{unsrt}
\bibliography{bibi}

\end{document}